\documentclass[conference]{IEEEtran}

\usepackage{cite}
\usepackage{amsmath,amssymb,amsfonts}
\usepackage{algorithmic}
\usepackage{graphicx}
\usepackage{textcomp}
\usepackage{xcolor}
\usepackage[affil-it]{authblk}
\def\BibTeX{{\rm B\kern-.05em{\sc i\kern-.025em b}\kern-.08em
    T\kern-.1667em\lower.7ex\hbox{E}\kern-.125emX}}

\newcommand{\cA}{{\cal A}}

\newcommand{\cC}{{\cal C}}

\newcommand{\cM}{{\cal M}}

\newcommand{\bldc}{{\mathbf{c}}}
\newcommand{\bldx}{{\mathbf{x}}}
\newcommand{\bldy}{{\mathbf{y}}}

\newcommand{\Type}[1]{{Type~${#1}$}}
\newcommand{\TA}{{\texttt{\large A}}}
\newcommand{\TB}{{\texttt{\large B}}}
\newcommand{\TC}{{\texttt{\large C}}}
\newcommand{\Tone}{{\mathrm{I}}}
\newcommand{\Ttwo}{{\mathrm{II}}}

\newcommand{\mmod}{{\mbox{mod}}}

\newcommand{\field}[1]{\mathbb{#1}}

\newcommand{\F}{\field{F}}

\newtheorem{theorem}{Theorem}
\newtheorem{lemma}{Lemma}
\newtheorem{corollary}{Corollary}
\newtheorem{proposition}{Proposition}

\begin{document}

\title{\title{Closed Expressions for the Weight Distributions\\
of Codes Associated with Perfect Codes}
}

\author{Tuvi Etzion\IEEEauthorrefmark{1},
        Denis Krotov\IEEEauthorrefmark{2}\thanks{D.S. Krotov is supported within the framework of the state contract of the Sobolev Institute of
Mathematics (FWNF-2026-0011).},
        Minjia Shi\IEEEauthorrefmark{3},
        and Wenhao Song\IEEEauthorrefmark{3}}

\affil{
\IEEEauthorrefmark{1}\small{Department of Computer Science, Technion --- Israel Institute of Technology, Haifa, 3200003 Israel}\\
\IEEEauthorrefmark{2}\small{Sobolev Institute of Mathematics, Novosibirsk 630090, Russia}\\
\IEEEauthorrefmark{3}\small{School of Mathematical Sciences, Anhui University, Hefei 230601, China}\\
}

\maketitle

\begin{abstract}
Perfect codes are arguably the most fascinating structures in combinatorial coding theory, and their
classification and weight distribution are of considerable interest.
This classification also involves the analysis of some related structures.
This paper considers five closely related structures, but all of them have never been tied together before.
These structures are 1-perfect codes, extended 1-perfect codes, nearly perfect 1-covering codes, extended nearly perfect 1-covering codes,
and one family of completely regular codes (to be called diamond codes).
The current work concentrates on the weight distributions of these five families of codes.
In the past, some of these weight distributions were not computed, some required heavy tools, and for some
only the weight enumerator was presented. We provide complete weight distributions for all five families using some methods
that do not require any heavy tools.
\end{abstract}


\section{Introduction}
\label{sec:Intro}

Perfect codes have always been fascinating structures, and there is extensive literature on them~\cite{Etz22}
for various schemes, metrics, and distance measures. However, the first and arguably the most interesting
scheme to consider is the Hamming scheme.

In this work only the binary Hamming scheme is considered, i.e., all codes are over the binary field $\F_2$
and an $(n,M)$ code~$\cC$ (of length $n$ and size $M$) is a subset of  $\F_2^n$.
The \emph{(Hamming) distance} between two words $\bldx,\bldy \in \F_2^n$  denoted by $d(\bldx,\bldy)$
is the number of coordinates with different values between $\bldx=(x_1,x_2,\ldots,x_n)$ and $\bldy=(y_1,y_2,\ldots,y_n)$.
The \emph{weight} of a word $\bldx$ is the number of nonzero entries in~$\bldx$.
The \emph{minimum distance} of a code $\cC$ is the smallest distance between any two distinct codewords in $\cC$.
The distance of a word $\bldx \in \F_2^n$ from $\cC$ is defined by
\vspace{-0.05cm}
$$
d(\bldx,\cC) = \text{min}_{\bldc \in \cC} d(\bldx,\bldc) .
$$
A related parameter, $R$, the \emph{covering radius} of $\cC$ is defined by
\vspace{-0.05cm}
$$
R=\text{max}_{\bldx \in \F_2^n} d(\bldx , \cC) ,
$$
i.e., $R$ is the distance of the most distant word from $\cC$.

A \emph{translate} of an $(n,M)$ code $\cC$ is the set
\vspace{-0.05cm}
$$
\bldx+ \cC \triangleq \{ \bldx+\bldc ~:~ \bldc \in \cC \},
$$
where $\bldx \in \F_2^n$ and the addition $\bldx+\bldc$ is done bitwise modulo~2.

For an $(n,M)$ code $\cC$ with minimum distance $2R+1$ we have the \emph{sphere-packing bound}
\vspace{-0.05cm}
$$
M \cdot  \sum_{i=0}^R \binom{n}{i} \leq 2^n ,
$$
and if $\cC$ has covering radius $R$ we have the \emph{sphere-covering bound}
\vspace{-0.05cm}
\begin{equation}
\label{eq:sp_cover}
M \cdot  \sum_{i=0}^R \binom{n}{i} \geq 2^n .
\end{equation}
A code that meets both bounds is called a \emph{perfect code}. The following natural result is a consequence
from this definition.
\begin{lemma}
\label{lem:Pradius}
If $\cC$ is a perfect code of length $n$, covering radius~$R$, and minimum distance $2R+1$, then for every word $\bldx \in \F_2^n$, there
exists exactly one codeword $\bldc \in \cC$ such that $d(\bldx,\bldc) \leq R$.
\end{lemma}

The only nontrivial infinite family of perfect codes are ${1\text{-perfect}}$ codes.
The length of such a code is $2^r -1$ and its number of codewords is $2^{2^r -r-1}$.
From now when a perfect code is mentioned it will be referred to such a code.

An improvement on the sphere-covering bound was obtained by van Wee~\cite{vWee88}.
A simplified version of this bound was presented by Struik~\cite{Str94,Str94a} and takes the form
\vspace{-0.05cm}
\begin{equation}
\label{eq:struik}
M \cdot \left( \sum_{i=0}^R \binom{n}{i} - \frac{\binom{n}{R}}{\left\lfloor \frac{n}{R+1} \right\rfloor} \left( \left\lceil \frac{n+1}{R+1} \right\rceil - \frac{n+1}{R+1} \right) \right) \geq 2^n .
\end{equation}
When $R+1$ divides $n+1$, this bound coincides with~(\ref{eq:sp_cover}).
A code that meets the bound~(\ref{eq:struik}) is called a \emph{nearly perfect covering code}. For even $n$ and $R=1$, this bound becomes
\vspace{-0.05cm}
\begin{equation}
\label{eq:vanWee}
M \geq \frac{2^n}{n} .
\end{equation}

Except for perfect codes and some trivial codes, $R=1$ is the only radius for which we currently know of codes that meet the bound~(\ref{eq:struik});
from~(\ref{eq:vanWee}), these codes have length $n=2^r$ and size $M=2^{2^r-r}$, for some integer $r \geq 2$.
A code with these parameters will be called a \emph{nearly perfect \textup{1}-covering code} (in short, NP1CC) and from now
throughout the paper we assume that $n=2^r$ and $M=2^{n-r}$.

It was proved in~\cite{BER25} that each codeword $\bldc$ of an NP1CC $\cC$ has a unique \emph{partner} $\bldc' \in \cC$ such that
either $d(\bldc,\bldc')=1$ or $d(\bldc,\bldc')=2$. This pair of codewords $\{ \bldc,\bldc'\}$ will be also called a \emph{pair of partners}
or a \emph{\Type{\Tone} pair} if $d(\bldc,\bldc')=1$ or a \emph{\Type{\Ttwo} pair} if $d(\bldc,\bldc')=2$.
For any pair of partners $\bldc$ and $\bldc'$ at distance two, there are exactly two words at distance one from
both $\bldc$ and $\bldc'$. These words will be called \emph{midwords}.

The NP1CCs are classified into three types:
\begin{itemize}
\item \Type{\TA} codes, in which the codewords in each pair of partners $\{ \bldc,\bldc'\}$ are \Type{\Tone} pairs.

\item \Type{\TB} codes, in which the codewords in each pair of partners $\{ \bldc,\bldc'\}$ are \Type{\Ttwo} pairs.

\item \Type{\TC} codes which are neither of \Type{\TA} nor of \Type{\TB}.
\end{itemize}

For an $(n',M')$ code $\cC$, the \emph{extended code} $\cC^*$ is $(n'+1,M')$ code obtained from $\cC$ by adding to each codeword of $\cC$
an even parity as an $(n'+1)$-st coordinate, i.e.,
\vspace{-0.05cm}
$$
\cC^* \triangleq \{ (\bldc,p(\bldc)) ~:~ \bldc \in \cC \},
$$
where $p(\bldc) = \sum_{i=1}^{n'} c_i ~ (\mmod ~ 2)$ is the (even) \emph{parity} of $\bldc$. This is an extended code of even weight.
Similarly, an extended code of odd weight is defined.
An extended NP1CC is called an ENP1CC.
The following simple result was proved in~\cite{BER25}.

\begin{lemma}
In an ENP1CC $C^*$, for each codeword $\bldc \in C^*$ there is a unique codeword $\bldc' \in \cC^*$ such that
$d(\bldc,\bldc')=2$ (and $d(\bldc,\bldc'') \geq 4$ for any other codeword $\bldc'' \in \cC^* \setminus \{ \bldc,\bldc' \}$).
\end{lemma}

Similar to NP1CCs also two codewords $\bldc,\bldc'$, in an $(n+1,M)$ ENP1CC, for which $d(\bldc,\bldc')=2$ will be called partners or a pair of partners.
The following result on NP1CCs was proved in~\cite{BER25}.
\begin{lemma}
\label{lem:TypeA}
An NP1CC of length $n=2^r$ is of \Type{\TA} if and only if it is a union of an even translate of an extended perfect code and an odd translate
of an extended perfect code.
\end{lemma}

There is one basic and simple construction for NP1CCs presented in the following theorem~\cite{BER25}.
\begin{theorem}
\label{thm:simpConst}
If $\cC_1$ and $\cC_2$ are $(n-1,M/2)$ perfect codes, then the code
\vspace{-0.05cm}
$$
\cC \triangleq \{ (\bldc_1,0) ~:~ \bldc_1 \in \cC_1 \} \cup \{ (\bldc_2,1) ~:~ \bldc_2 \in \cC_2 \}
$$
is an $(n,M)$ NP1CC.
\end{theorem}

A code is a \emph{zeroed code} if the all-zero word is a codeword.
There is a simple correspondence between zeroed perfect codes and zeroed extended perfect codes
(adding an even parity to a zeroed perfect code and puncturing a coordinate of a zeroed extended perfect code).
A code $\cC$ of length $m$ has a weight distribution $\cA(\cC)=(A_0,A_1,\ldots,A_m)$, where $A_i=\cA_i(\cC)$ is the number of codewords
of weight $i$ in~$\cC$.

There are tight connections between perfect codes, extended perfect codes, NP1CCs, ENP1CCs, and diamond codes.
In this paper we concentrate only on the tight connections between ENP1CCs and diamond codes. Other connections, mainly
the type of NP1CCs obtained from one ENP1CC are explored and almost completely determined in~\cite{EKSS}, where also
proofs, omitted here, will be given.

The rest of the paper is organized as follows.
The family of diamond codes, which is equivalent to the family
of ENP1CCs, is discussed in Section~\ref{sec:midwordsENP}.
In Section~\ref{sec:weight_distribute}, we present closed expressions for the weight distributions of the five families that are discussed.
Moreover, we show the uniqueness of these weight distributions in a simpler way than what was shown before for only some of these families.

\section{ENP1CCs and Diamond Codes}
\label{sec:midwordsENP}

In this section, we consider one class of another type of codes called completely regular codes.
The family of perfect codes is contained in the family of completely regular codes, i.e., completely regular codes generalize the
concept of perfect codes.
Let $s_{1,1}$, $s_{1,2}$, $s_{2,1}$, and $s_{2,2}$ -- be the following four values defined for a code (with covering radius 1 for our requirements)
of length $m$.
\begin{itemize}
\item $s_{1,1}$ is the number of codewords that are at distance one from a given codeword $\bldc$.

\item $s_{1,2}$ is the number of non-codewords at distance one from a given codeword $\bldc$.

\item $s_{2,1}$ is the number of codewords at distance one from a given non-codeword $\bldx$.

\item $s_{2,2}$ is the number of non-codeword at distance one from a given non-codeword $\bldx$.
\end{itemize}
A nonempty code with a nonempty complement is called \emph{completely regular code}
(for short, we omit ``covering radius~1" in this paper) if $s_{1,1}$, $s_{1,2}$, $s_{2,1}$, and $s_{2,2}$ do not
depend on the particular choice of the codeword $\bldc$ and the non-codeword $\bldx$. Completely regular codes with
arbitrary covering radius are not defined here, but they are extensively studied;
a survey on such codes can be found in~\cite{KrPo25}.
The matrix
\vspace{-0.05cm}
$$
S= \left(\begin{array}{cc}
s_{1,1} & s_{1,2} \\
s_{2,1} & s_{2,2}
\end{array}\right)
$$
is called the \emph{quotient matrix} of the code.

A completely regular code with quotient matrix
\vspace{-0.05cm}
$$
\Diamond =\left(\begin{array}{cc} 2 & m-2 \\ 1 & m-1 \end{array}\right)
$$
is called a \emph{diamond code} or a \emph{$\Diamond$-CR code}.

Although it was stated in the definition, for completeness, we state the following fact
that is observed from the definition of the quotient matrix.
\begin{lemma}
\label{lem:SCRcover}
A diamond code $\hat{\cC}$ is a covering code with radius~$1$.
\end{lemma}

Let $\cC^*$ be an ENP1CC and let $\cM(\cC^*)$ be the set of midwords of all the partners in $\cC^*$.

\begin{lemma}
\label{lem:paramSCR}
If $\hat{\cC}$ is an $(m,M')$ diamond code, then $m=n+1$ and $M'=2M$.
\end{lemma}
\begin{IEEEproof}
Let $\hat{\cC}$ be an $(m,M')$ diamond code. Since $M'$ is the number of codewords in $\hat{\cC}$, it follows that
the number of non-codewords in $\F_2^m$ is $2^m - M'$. By the quotient matrix, each non-codeword is at distance one
from exactly one codeword and there are $m-2$ non-codewords at distance one from a given codeword.
This implies that the number of non-codewords in $\F_2^m$ is $(m-2)M'$, i.e.,
\vspace{-0.05cm}
$$
2^m - M' = (m-2) M' ~~ \text{i.e.,} ~~ 2^m = (m-1) M' ~.
$$
Hence, both $m-1$ and $M'$ are powers of two. Therefore, the only solutions for $m$ and $M'$, for any $r \geq 2$, are
$m=2^r+1=n+1$ and $M' = 2^{2^r +1 -r} =2M$.
\end{IEEEproof}

\begin{proposition}
\label{prop:OTO_ENP_CR}
There is a one-to-one correspondence between the set of zeroed $(n+1,M)$ ENP1CCs and the set of zeroed $(n+1,2M)$ diamond codes.
\end{proposition}
\begin{IEEEproof}
If $\cC^*$ is an ENP1CC and $\cM(\cC^*)$ is the set of its midwords, then we form a new code $\hat{\cC}= \cC^* \cup \cM (\cC^*)$.
It can be proved, and it will be done in~\cite{EKSS}, that the quotient matrix of $\hat{\cC}$ is $\Diamond$ and hence $\hat{\cC}$ is an diamond code.

Assume $\hat{\cC}$ is an $(n+1,2M)$ diamond code. By the quotient matrix $\Diamond$, each codeword has two neighbouring
codewords at distance one. Hence, we can order the set of codewords in disjoint cycles of even length, where two adjacent codewords in the cycles have
distance one. Let $\bldc_1$, $\bldc_2$, and $\bldc_3$ be three consecutive codewords on such a cycle,
i.e., $d(\bldc_1,\bldc_2)=d(\bldc_2,\bldc_3)=1$ and $d(\bldc_1,\bldc_3)=2$, which implies that there exists another word $\bldx$ such that
$d(\bldc_1,\bldx)=1$ and $d(\bldc_3,\bldx)=1$. If $\bldx$ is a non-codeword, then it is at distance one from two codewords, contradicting
the quotient matrix $\Diamond$ (each non-codeword is at distance one from exactly one codeword). If there exist two codewords
$\bldc_1$ and~$\bldc_2$ such that $d(\bldc_1,\bldc_2)=1$, with no other codewords at distance one from them, it contradicts
the quotient matrix $\Diamond$ (each codeword is at distance one from exactly two codewords).
Hence, all the cycles have length $4$, and each cycle has two codewords of even weight and two codewords of odd weight.
Thus, $\hat{\cC}$ has $M$ codewords of even weight and $M$ codewords of odd weight.

Therefore, we can form a code~$\cC^*$ from the codewords of even weight of $\hat{\cC}$.
In~$\cC^*$, there are no two codewords from different cycles of~$\hat{\cC}$ whose distance is two, since this will imply a non-codeword at distance one
from two codewords of~$\hat{\cC}$, contradicting the definition of the quotient matrix.
Thus, $\cC^*$ is an $(n+1,M)$ ENP1CC.
\end{IEEEproof}

Similar to Proposition~\ref{prop:OTO_ENP_CR} we can prove one-to-one correspondence between diamond codes and ENP1CCs
with various combinations of codewords with the smallest weight.

\begin{corollary}
\label{cor:CodeCycles}
The codewords of an $(n+1,2M)$ diamond code form cycles of length $4$, where each cycle contains one codeword of weight~$k$, for some $0 \leq k \leq n-1$,
two codewords of weight $k+1$ and one codeword of weight $k+2$. Each codeword not on a given cycle is at distance at least $3$ from
the codewords of the cycle.
\end{corollary}

While in Section~\ref{sec:weight_distribute} complete weight distributions for the five families are given, we can already
present two results on the weight distribution of diamond codes. These two results will be proved in~\cite{EKSS}.

\begin{theorem}
\label{thm:WD_SR}
Zeroed diamond codes have a unique weight distribution and non-zeroed translates of diamond codes also have a unique
weight distribution.
\end{theorem}

\begin{lemma}
\label{lem:Mid_comp}
An ENP1CC $\cC^*$ and the set $\cM (\cC^*)$ of its midwords are complement to each other.
\end{lemma}
%

Proposition~\ref{prop:OTO_ENP_CR}, Theorem~\ref{thm:WD_SR}, and Lemma~\ref{lem:Mid_comp}
induce the connections between zeroed and non-zeroed $(n+1,M)$
ENP1CCs and zeroed and non-zeroed $(n+1,2M)$ diamond codes. The analysis is straightforward and hence omitted.

\begin{corollary}
\label{cor:CReqENP}
$~$
\begin{itemize}
\item A zeroed $(n+1,2M)$ diamond code $\hat{\cC}$ is obtained from a zeroed $(n+1,M)$ ENP1CC $\cC^*$ by adding to $\cC^*$ all the midwords of~$\cC^*$,
i.e., $\hat{\cC} = \cC^* \cup \cM(\cC^*)$ or by adding the complements of $\cC^*$, i.e., $\hat{\cC} = \{ \bar{\bldc} ~:~ \bldc \in \cC^* \}$.

\item A non-zeroed $(n+1,2M)$ diamond code $\hat{\cC}$ is obtained from a non-zeroed $(n+1,M)$ ENP1CC $\cC^*$ with one codeword of weight one by
adding to $\cC^*$ all the midwords of~$\cC^*$, or by adding the complements of $\cC^*$.

\item A zeroed $(n+1,M)$ ENP1CC $\cC^*$ is obtained from a zeroed $(n+1,M)$ diamond code $\hat{\cC}$ by taking
all the even weight codewords of a zeroed $(n+1,2M)$ diamond.

\item A non-zeroed $(n+1,M)$ ENP1CC with two codewords of weight one is obtained from the odd weight codewords of a zeroed $(n+1,2M)$ diamond code.

\item A non-zeroed $(n+1,M)$ ENP1CC with one codeword of weight one is obtained from the odd weight codewords of a non-zeroed $(n+1,2M)$ diamond code.

\item A non-zeroed $(n+1,M)$ ENP1CC with even weight codewords is obtained from the even weight codewords of a non-zeroed $(n+1,2M)$ diamond code.
\end{itemize}
\end{corollary}

\section{Weight Distributions for the Five Families}
\label{sec:weight_distribute}

In this section, we derive closed expressions of the weight distribution for the five families of codes considered in this paper and
for their related translates. Some of these weight distributions are known and were derived before; for some, the weight enumerator
was computed, and some were never computed and published before. As the weight distribution of codes and their translates
is one of the most important parameters of a code, it is worthwhile to have all the weight distributions of these five
families of codes in one place. This section does exactly this, it brings all these weight distributions in closed expressions together
into one place. The first result can be deduced from Corollary~\ref{cor:CReqENP}.

\begin{theorem}
\label{thm:WD_CR_E}
$~$
\begin{itemize}
\item There is a unique weight distribution for zeroed diamond codes and a unique weight distribution for non-zeroed diamond codes.

\item There is a unique weight distribution for zeroed ENP1CCs. There are three possible weight distributions a non-zeroed
ENP1CC $\cC^*$, one for which $\cA_0 (\cC^*)=0$, $\cA_1 (\cC^*)=1$, and $\cA_2 (\cC^*)=0$,
one for which  $\cA_0 (\cC^*)=0$, $\cA_1 (\cC^*)=2$, and $\cA_2 (\cC^*)=0$, and one for which $\cA_0 (\cC^*)=0$, $\cA_1 (\cC^*)=0$,
and $\cA_2 (\cC^*)=\frac{n}{2} +1$.
\end{itemize}
\end{theorem}
\begin{IEEEproof}
$~$
\begin{itemize}
\item The claims regarding diamond codes are taken from Theorem~\ref{thm:WD_SR}.

\item By Proposition~\ref{prop:OTO_ENP_CR} there is a one-to-one correspondence between the set of ENP1CCs and the set
of diamond codes. A zeroed diamond code $\hat{\cC}$ has a unique weight distribution, where $\cA_0 (\hat{\cC}) =1$
which implies that $\cA_1 (\hat{\cC}) =2$, and $\cA_2 (\hat{\cC}) =1$ (see the proofs of Theorem~\ref{thm:WD_SR} and Corollary~\ref{cor:CReqENP}).
This implies a zeroed $(n+1,M)$ ENP1CC $\cC^*$ for which $\cA_0 (\cC^*)=1$, $\cA_1 (\cC^*)=0$, and $\cA_2 (\cC^*)=1$
and a translated non-zeroed $(n+1,M)$ ENP1CC $\cC^*$ for which $\cA_0 (\cC^*)=0$, $\cA_1 (\cC^*)=2$, and $\cA_2 (\cC^*)=0$.
A non-zeroed diamond code $\hat{\cC}$ also has a unique weight distribution, where by the proofs of Theorem~\ref{thm:WD_SR}
and Corollary~\ref{cor:CReqENP},
$\cA_0 (\hat{\cC})=0$, $\cA_1 (\hat{\cC})=1$, and $\cA_2 (\hat{\cC})=\frac{n}{2} +1$.
This implies a translated non-zeroed $(n+1,M)$ ENP1CC $\cC^*$ for which
$\cA_0 (\cC^*)=0$, $\cA_1 (\cC^*)=1$, and $\cA_2 (\cC^*)=0$; and another translated $(n+1,M)$ ENP1CC $\cC^*$ for which
$\cA_0 (\cC^*)=0$, $\cA_1 (\cC^*)=0$, and $\cA_2 (\cC^*)=\frac{n}{2} +1$. Thus, the claim of the theorem is proved.
\end{itemize}
\end{IEEEproof}

\begin{theorem}
\label{thm:uniqueNP1CC}
There is a unique weight distribution for zeroed $(n,M)$ NP1CCs of \Type{\TA} which can also be the
weight distribution of a zeroed NP1CC of \Type{\TC}.
There is a unique weight distribution for zeroed $(n,M)$ NP1CCs of \Type{\TB} which can also be the
weight distribution of a zeroed NP1CC of \Type{\TC}.

There is a unique weight distribution for non-zeroed $(n,M)$ NP1CCs of \Type{\TA} which can
also be the weight distribution of non-zeroed $(n,M)$ NP1CCs of \Type{\TB} or of \Type{\TC} when they contain exactly one word of weight one.
There is a weight distribution for non-zeroed $(n,M)$ NP1CCs of \Type{\TB} and
for non-zeroed $(n,M)$ NP1CCs of \Type{\TC} when they contain two words of weight one.
\end{theorem}
\begin{IEEEproof}
Since each codeword in an NP1CC has a unique partner at either distance one or at distance two, it follows
that the first three possible values in the weight distribution of a zeroed $(n,M)$ NP1CC $\cC$ are:
\newline
$\cA_0(\cC) =1$, $\cA_1 (\cC)=1$, and $\cA_2 (\cC)=0$ --- will be denoted by WD1.
\newline
$\cA_0(\cC) =1$, $\cA_1 (\cC)=0$, and $\cA_2 (\cC)=1$ --- will be denoted by WD2.

The extended code of a non-zeroed $(n,M)$ NP1CC $\cC$
can have three weight distributions (see Theorem~\ref{thm:WD_CR_E}),
and hence the first three possible values in its weight distribution are:
\newline
$\cA_0(\cC) =0$, $\cA_1 (\cC)=1$, and $\cA_2 (\cC)=\frac{n}{2}$ --- will be denoted by WD3.
\newline
$\cA_0(\cC) =0$, $\cA_1 (\cC)=2$, and $\cA_2 (\cC)=\frac{n}{2} -1$ --- will be denoted by WD4.

Our goal is to prove that each one of these three values for the weight distribution determines it uniquely.
On the other hand, by Theorem~\ref{thm:WD_CR_E}, a zeroed $(n+1,M)$ ENP1CC $\cC^*$ has a unique weight distribution given by
\newline
$\cA_0(\cC^*) = 1$, $\cC_1(\cC^*)=0$, and $\cA_2(\cC^*)=1$ --- will be denoted by WD5.

Similarly, by Theorem~\ref{thm:WD_CR_E} a non-zeroed $(n+1,M)$ ENP1CC $\cC^*$ has exactly three possible
weight distributions and the first three possible values in these weight distributions are
\newline
$\cA_0(\cC^*) = 0$, $\cA_1(\cC^*)=1$, and $\cA_2(\cC^*)=0$ --- will be denoted by WD6.
\newline
$\cA_0(\cC^*) = 0$, $\cA_1(\cC^*)=2$, and $\cA_2(\cC^*)=0$ --- will be denoted by WD7.
\newline
$\cA_0(\cC^*) = 0$, $\cA_1(\cC^*)=0$, and $\cA_2(\cC^*)=\frac{n}{2} +1$ --- will be denoted by WD8.

There is a simple connection between the weight distribution of the NP1CC $\cC$ and the weight distribution
of its ENP1CC $\cC^*$ as follows. For even $i$, and even weight ENP1CC $\cC^*$ (zeroed or non-zeroed),
$\cA_0 (\cC^*) = \cA_0 (\cC)$, and for $i > 0$,
\vspace{-0.05cm}
\begin{equation}
\label{eq:evenEWD}
\cA_i (\cC^*) = \cA_{i-1} (\cC) + \cA_i (\cC).
\end{equation}
For odd $i$, and odd weight non-zeroed ENP1CC $\cC^*$,
\vspace{-0.05cm}
\begin{equation}
\label{eq:oddEWD}
\cA_i (\cC^*) = \cA_{i-1} (\cC) + \cA_i (\cC).
\end{equation}

Based on Equations~(\ref{eq:evenEWD}) and~(\ref{eq:oddEWD}), we analyze these eight weight distributions as follows.

WD1 yields only the odd weight distribution WD7 and the even weight distribution WD5.
This implies that the first three values of the weight distribution WD1, Equations~(\ref{eq:evenEWD}) and~(\ref{eq:oddEWD}),
and the uniqueness of the four weight distributions for zeroed and non-zeroed ENP1CCs, yields a determined weight distribution
that starts with the three values of WD1.

Similar analysis is done for WD1, WD2, and WD3~\cite{EKSS}.

This analysis and the possible matchings of the weight distributions WD1, WD2, WD3, and WD4, to zeroed and non-zeroed NP1CCs
of \Type{\TA}, \Type{\TB}, \Type{\TC}, imply the claim of the theorem.
\end{IEEEproof}

Theorem~\ref{thm:uniqueNP1CC} has a relatively short and simple proof replacing the proof in~\cite{BER25}
which was quite complicated and used much heavier tools outlined in Chapters 5 and 6 in~\cite{McSl77}.

The weight distribution of a perfect code and a translate of a perfect code was determined in~\cite{EtVa94}
based on a solution for the recursion on the weight distribution of an $(n-1,M/2)$ perfect code and its translates.
If $\cC$ is an $(n-1,M/2)$ perfect code of length $n-1=2^r-1 = 2 \nu +1$, i.e., $\nu = 2^{r-1} -1$, then
\vspace{-0.05cm}
\begin{equation}
\label{eq:WD_perfect}
\cA_i(\cC) = \frac{\binom{n-1}{i} + (n-1) \Delta_i }{n} \triangleq \alpha_i,
\end{equation}
where
\vspace{-0.05cm}
\begin{equation}
\label{eq:sol_sum_w}
\Delta_i=\begin{cases}
          \binom{\nu}{\lfloor i/2 \rfloor} & ~~ i \equiv 0,3~(\textup{\mmod}~4)\\
          -\binom{\nu}{\lfloor i/2 \rfloor} & ~~i \equiv 1,2~(\textup{\mmod}~4)
         \end{cases}~.
\end{equation}
If $\cC$ is an odd translate of an $(n-1,M/2)$ perfect code of length $n-1=2^r-1$, then
\vspace{-0.05cm}
\begin{equation}
\label{eq:WD_perfect_trans}
\cA_i(\cC) = \frac{\binom{n-1}{i} - \Delta_i }{n} \triangleq \beta_i,
\end{equation}

Note that $n=2\nu +2$ and for odd $i$,
\vspace{-0.05cm}
\begin{equation}
\label{eq:zeroDelta}
\Delta_i + \Delta_{i-1} =0
\end{equation}
and for even $i$ we have
\vspace{-0.05cm}
\begin{equation}
\label{eq:non_zeroDelta}
\Delta_i + \Delta_{i-1} =
        \begin{cases}
          \binom{\nu+1}{i/2} & ~~ i \equiv 0~(\textup{\mmod}~4)\\
          -\binom{\nu+1}{i/2} & ~~i \equiv 2~(\textup{\mmod}~4)
        \end{cases}~.
\end{equation}

The equalities of Equations~(\ref{eq:zeroDelta}) and~(\ref{eq:non_zeroDelta}) will be used throughout this section.
The weight distribution of a perfect code given in~(\ref{eq:WD_perfect}) and in~(\ref{eq:WD_perfect_trans}), and the value
of $\Delta_i$ in~(\ref{eq:sol_sum_w}) imply the weight distributions of an extended perfect code
and also of translates of the extended perfect code as follows.
\begin{theorem}
$~$
\begin{enumerate}
\item If $\cC$ is a zeroed $(n,M/2)$ extended perfect code, then $\cA_0(\cC)=1$, $\cA_i(\cC)=0$ for odd $i$,
and for even $i >0$ we have
\vspace{-0.07cm}
\begin{equation*}
\label{eq:WD_EXH}
\cA_i (\cC)= \alpha_i + \alpha_{i-1} = \frac{\binom{n}{i} + (n-1)(\Delta_i + \Delta_{i-1})}{n}\\
\end{equation*}

\item If $\cC$ is an odd translate of an $(n,M/2)$ extended perfect code, then $\cA_i(\cC)=0$ for even $i$ and for odd $i$ we have
\vspace{-0.05cm}
\begin{equation*}
\label{eq:WD_EXH_O}
\cA_i (\cC)= \frac{\binom{n}{i} + (n-1)(\Delta_i + \Delta_{i-1})}{n}  =\frac{1}{n} \binom{n}{i}. 
\end{equation*}

\item If $\cC$ is an even non-zeroed $(n,M/2)$ extended perfect code, then $\cA_i(\cC)=0$ for odd $i$,
$\cA_0 (\cC)=0$, and for even $i >0$ we have
\vspace{-0.05cm}
\begin{equation*}
\label{eq:WD_EXH_E}
\cA_i (\cC)=\beta_i + \beta_{i-1} = \frac{\binom{n}{i} - (\Delta_i + \Delta_{i-1})}{n}~.
\end{equation*}
\end{enumerate}
\end{theorem}

The unique four weight distributions for zeroed NP1CCs and their translates as implied by Theorem~\ref{thm:uniqueNP1CC},
the weight distribution for
an extended perfect code and its translates of even and odd weights imply the weight distribution of zeroed $(n,M)$ NP1CCs
of \Type{\TA} and zeroed $(n,M)$ NP1CCs of \Type{\TC} that contain a codeword of weight one.
The proof of this weight distribution is based on Lemma~\ref{lem:TypeA} and Theorem~\ref{thm:uniqueNP1CC}.

\begin{theorem}
\label{thm:wd_npA}
$~$
\begin{enumerate}
\item If $\cC$ is a zeroed $(n,M)$ NP1CC of \Type{\TA} or a zeroed $(n,M)$ NP1CC of \Type{\TC} with a codeword of weight one, then
$
\cA_0(\cC)=\cA_{n} (\cC)=1.
$
For other even $i$,
\vspace{-0.05cm}
$$
\cA_i (\cC) = \alpha_i + \alpha_{i-1}= \frac{\binom{n}{i} + (n-1) (\Delta_{i-1} +\Delta_i)}{n} ~.
$$
\vspace{-0.05cm}
$$
\text{For~odd} ~ i, ~1 \leq i \leq n-1,~
\cA_i (\cC) = \frac{1}{n} \binom{n}{i} .
$$

\item If $\cC$ is a non-zeroed $(n,M)$ NP1CC of \Type{\TA} or a non-zeroed $(n,M)$ NP1CC, of either \Type{\TB} or \Type{\TC},
that contains exactly one codeword of weight one, then
\vspace{-0.05cm}
$$
\cA_0(\cC)=\cA_{n} (\cC)=0.
$$
For even $i$, $1 < i < n-1$,
\vspace{-0.05cm}
$$
\cA_i (\cC) = \beta_i + \beta_{i-1} = \frac{\binom{n}{i} - (\Delta_{i-1} +\Delta_i)}{n} ~.
$$
\vspace{-0.05cm}
$$
\text{For~odd} ~ i, ~ 1 \leq i \leq n-1, ~
\cA_i (\cC) = \frac{1}{n} \binom{n}{i} .
$$
\end{enumerate}
\end{theorem}

\begin{theorem}
\label{thm:wd_npB}
$~$
\begin{enumerate}
\item If $\cC$ is a zeroed $(n,M)$ NP1CC of \Type{\TB} or a zeroed $(n,M)$ NP1CC of \Type{\TC} that contains a codeword of weight two, then
$
\cA_0(\cC)=1,
$
and for $i$, $1 \leq i \leq n$,
\vspace{-0.05cm}
$$
\cA_i (\cC) = \frac{\binom{n}{i} + (n-1)\Delta_i -\Delta_{i-1}}{n} ~.
$$

\item If $\cC$ is a translate of an $(n,M)$ NP1CC of \Type{\TB} that contains two codewords of weight one, or
a translate of an $(n,M)$ NP1CC of \Type{\TC} that contains two codewords of weight one, then
$
\cA_0(\cC)=0, ~~ \cA_1(\cC)=2,
$
and for $i \geq 2$,
\vspace{-0.08cm}
$$
\cA_i (\cC) = \frac{\binom{n}{i} + (n-1)\Delta_{i-1} -\Delta_i}{n} ~.
$$
\end{enumerate}
\end{theorem}
To the information given in Theorem~\ref{thm:uniqueNP1CC} we add the
weight distributions for ENP1CCs, diamond codes, and their translates.

\begin{theorem}
\label{thm:WD_ENP1CC}
$~$
\begin{enumerate}
\item If $\cC^*$ is a zeroed $(n+1,M)$ ENP1CC, then $\cA_i (\cC^*)=0$ for odd $i$,
$\cA_0(\cC^*)=\cA_{n} (\cC^*)=1$, and for even $i > 0$,
\vspace{-0.05cm}
$$
\cA_i(\cC^*) = \frac{\binom{n+1}{i} + (n-1)(\Delta_{i-1} +\Delta_i)}{n}.
$$

\item If $\cC^*$ is a non-zeroed $(n+1,M)$ ENP1CC with exactly one codeword of weight one, then $\cA_i (\cC^*)=0$ for even~$i$,
$\cA_1(\cC^*)=1$, and for odd $i>1$,
\vspace{-0.05cm}
$$
\cA_i (\cC^*) = \frac{\binom{n+1}{i} - (\Delta_{i-1}+\Delta_{i-2} ) }{n} ~.
$$

\item If $\cC^*$ is a non-zeroed $(n+1,M)$ ENP1CC with exactly two codewords of weight one, then $\cA_i(\cC^*)=0$ for even~$i$,
$\cA_1 (\cC^*)=2$, and for odd $i>1$,
\vspace{-0.05cm}
$$
\cA_i (\cC^*) = \frac{\binom{n+1}{i} - (n-1)(\Delta_{i-1}+\Delta_{i-2})   }{n} ~.
$$

\item If $\cC^*$ is a non-zeroed $(n+1,M)$ ENP1CC with even weight codewords, then $\cA_i(\cC^*)=0$ for odd~$i$, $\cA_0(\cC^*)=0$,
and for even $i >0$,
\vspace{-0.05cm}
$$
\cA_i (\cC^*) = \frac{\binom{n+1}{i} - (\Delta_{i}+\Delta_{i-1} ) }{n} ~.
$$
\end{enumerate}
\end{theorem}

The proof of the next theorem is based on Corollary~\ref{cor:CReqENP} and Theorem~\ref{thm:WD_ENP1CC} or on the complements of codewords
of an ENP1CCs (see Lemma~\ref{lem:Mid_comp}).

\begin{theorem}
$~$
\begin{enumerate}
\item If $\hat{\cC}$ is a zeroed $(n+1,2M)$ diamond code, then
\vspace{-0.05cm}
$$
\cA_0(\hat{\cC})=\cA_{n+1} (\hat{\cC})=1,~~ \cA_1(\hat{\cC})=\cA_n(\hat{\cC})=2.
$$
For even $i >0$,
\vspace{-0.05cm}
$$
\cA_i(\hat{\cC}) = \frac{\binom{n+1}{i} + (n-1)(\Delta_{i-1} +\Delta_i)}{n}.
$$
and for odd $i >1$, $\cA_i(\hat{\cC}) = \cA_{n+1-i}(\hat{\cC})$, i.e.,
\begin{equation*}
\cA_i (\hat{\cC}) = \frac{\binom{n+1}{i} - (n-1)(\Delta_{i-1}+\Delta_{i-2})   }{n} ~.
\end{equation*}

\item If $\hat{\cC}$ is a non-zeroed $(n+1,2M)$ diamond code, then $\cA_0 (\hat{\cC})=0$ and $\cA_1 (\hat{\cC}) =1$.
For even $i >0$ we have
\vspace{-0.07cm}
$$
\cA_i(\hat{\cC}) = \frac{\binom{n+1}{i} - (\Delta_{i}+\Delta_{i-1} ) }{n} ~.
$$
For odd $i >1$ we have
\vspace{-0.05cm}
$$
\cA_i (\hat{\cC}) = \frac{\binom{n+1}{i} - (\Delta_{i-1}+\Delta_{i-2} ) }{n} ~.
$$

\end{enumerate}
\end{theorem}

\end{document}